\documentclass[11pt,oneside]{amsart}

\usepackage[a4paper]{geometry}
\usepackage[english]{babel}
\usepackage[latin1]{inputenc}
\usepackage[T1]{fontenc}
\usepackage{lmodern}

\usepackage{amsfonts,amssymb}
\usepackage{amsmath}
\usepackage{enumerate}

\usepackage{tikz}
\usepackage[numbers,sort]{natbib}

\newcommand{\Z}{\mathbb{Z}}
\newcommand{\N}{\mathbb{N}}
\newcommand{\PP}{\mathbb{P}}
\newcommand{\EE}{\mathbb{E}}
\newcommand{\RR}{\mathbb{R}}
\newcommand{\CC}{\mathbb{C}}

\newcommand{\X}{X}

\newcommand{\bigO}{{O}}
\newcommand{\Ecal}{\mathcal{E}}

\newcommand{\ovomega}{\overline{\omega}}
\newcommand{\loewleq}{\preceq}

\DeclareMathOperator{\Hess}{Hess}

\newcommand{\sca}[1]{\left\langle #1 \right\rangle} 
\newcommand{\inp}[1]{\langle #1 \rangle} 

\newtheorem{lemma}{Lemma}[section]
\newtheorem{theorem}[lemma]{Theorem}
\newtheorem{proposition}[lemma]{Proposition}
\newtheorem{corollary}[lemma]{Corollary}

\newtheorem*{definition}{Definition}
\newtheorem*{remark}{Remark}

\begin{document}

\title{Partitions of large unbalanced bipartites}
\author[Julien Bureaux]{Julien Bureaux}
\address{Laboratoire Modal'X, Université Paris Ouest Nanterre La Défense\\
    200 avenue de la République, 92\,000 Nanterre}
\email{julien.bureaux@u-paris10.fr}


\begin{abstract}
    We compute the asymptotic behaviour of the number of partitions of large vectors $(n_1,n_2)$ of $\mathbb{Z}_+^2$ in the critical regime $n_1 \asymp \sqrt{n_2}$ and in the subcritical regime $n_1 = o(\sqrt{n_2})$. This work completes the results established in the fifties by Auluck, Nanda, and Wright.
\end{abstract}

\maketitle

\section{Introduction}

How many ways are there to decompose a finite-dimensional vector whose components are
non-negative integers as a sum of non-zero vectors of the same kind, up to permutation
of the summands?
The celebrated theory of integer partitions deals with the one-dimensional version of this problem. We refer the reader to the monograph of \citet{andrews_theory_1998} for an account on the subject. In a famous paper, \citet{hardy_asymptotic_1918} discovered the asymptotic behaviour of the number of partitions of a large integer $n$:
\begin{equation}
    \label{eq:hardy-ramanujan}
    p_{\N}(n) \sim \frac{1}{4n\sqrt{3}}\exp\left\{\pi\sqrt{\frac{2n}{3}}\right\}.
\end{equation}

In the two-dimensional setting, the number of partitions of a large vector has been studied by many authors, by various approaches. In the early fifties, the physicist \citet{fermi_angular_1951} introduced thermodynamical models characterized by the conservation of two parameters instead of just one (corresponding to integer partitions). This led \citet{auluck_partitions_1953} to search for an asymptotic expression of the number of partitions
of an integer vector $(n_1,n_2)$. He established formul\ae{} in two very different regimes. His first formula holds when $n_1$ is fixed and $n_2$ tends to infinity:
\begin{equation}
    \label{eq:auluck1}
	p_{\Z_+^2}(n_1,n_2) \sim \frac{1}{n_1!}\left(\frac{\sqrt{6n_2}}{\pi}\right)^{n_1} \frac{1}{4n_2\sqrt{3}} \exp\left\{\pi\sqrt{\frac{2n_2}{3}}\right\}.
\end{equation}
His second one concerns the case where both components $n_1$ and $n_2$ tend to infinity with the same order of magnitude. The corresponding formula is much more involved but it can be simplified in the special case $n_1=n_2=n$. For some explicit constants $a,b,c$, one has
\[
	p_{\Z_+^2}(n,n) \sim \frac{a}{n^\frac{55}{36}} \exp\left\{b\, n^\frac{2}{3} + c\, n^\frac{1}{3}\right\}.
\]
In the late fifties, \citet{nanda_bipartite_1957} managed to extend the domain of validity of Auluck's first formula to the weaker condition $n_1=o(n_2^{1/4})$. Shortly after, \citet{wright_partitions_1958} was able to prove that Auluck's second formula can be extended to the more general regime
\begin{equation*}
    \frac{1}{2} < \liminf \frac{\log n_1}{\log n_2} \leq \limsup \frac{\log
n_1}{\log n_2} < 2.
\end{equation*}
Note finally that \citet*{robertson_asymptotic_1960,robertson_partitions_1962} proved analogous formulae{} in higher dimensions.

   Our article covers the case $n_1 = \bigO(\sqrt{n_2})$, which completes these previous results. In particular, we deal with the case where $n_1$ and $\sqrt{n_2}$ have the same order of magnitude, which appears as a critical regime.

   The papers of Auluck, Nanda and Wright all rely on generating function techniques. At the exception of Nanda's work, which is directly based on integer partition estimates, the main idea is to extract the asymptotic behaviour of the coefficients from the generating function with a Tauberian theorem or a saddle-point analysis. The extension proven by Wright was made possible by a more precise approximation of the generating function.

   \begin{figure}[h]
    \label{fig:regions_diagram}
    \centering
\begin{tikzpicture}
    \draw [->] (0,0) -- (5,0) ;
    \draw (5,0) node[right] {$\log n_1$} ;
    \draw [->] (0,0) -- (0,5) ;
    \draw (0,5) node[above] {$\log n_2$} ;
    \fill [color=gray!10] (0,0) -- (5,1.25) -- (5,2.5) -- cycle ;
    \fill [color=gray!10] (0,0) -- (1.25,5) -- (2.5,5) -- cycle ;
    \draw [dashed,ultra thin] (0,0) -- (2.5,2.5) ;
    \draw [dashed,ultra thin] (5,5) -- (3.5,3.5) ;
    \draw (5,5) node[above right] {$1$} ;
    \draw [ultra thick] (0,0) -- (5,2.5) ;
    \draw (5,2.5) node[right] {$\dfrac{1}{2}$} ;
    \draw [ultra thick] (0,0) -- (2.5,5) ;
    \draw (2.5,5) node[above] {$2$} ;
    \draw (0,0) -- (5,1.25) ;
    \draw (5,1.25) node[right] {$\dfrac{1}{4}$} ;
    \draw (0,0) -- (1.25, 5) ;
    \draw (1.25,5) node[above] {$4$} ;
    \draw (3,3) node {\citet{wright_partitions_1958}} ;
    \draw (4,0.5) node {\citet{nanda_bipartite_1957}} ;
    \draw (0.5,4) node {N.} ;
    \draw (4,1.5) node {?} ;
    \draw (1.5,4) node {?} ;
\end{tikzpicture}
\caption{Phase diagram of the previously studied asymptotic regions. Our work covers the unknown grey region, including the thick critical lines $\frac{1}{2}$ and $2$, as well as Nanda's region.}
\end{figure}
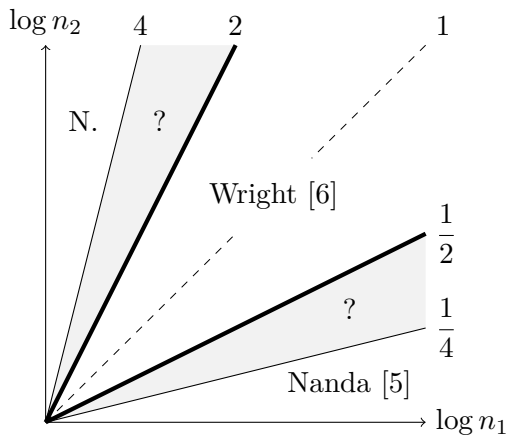

   The method we use in this article differs from the previous ones by relying heavily on a probabilistic embedding of the problem which is inspired by the Boltzmann model in statistical mechanics. The first ingredient of our proof is a precise estimate of the associated logarithmic partition function, based on an contour-integral representation of this function and Cauchy's residue theorem. The second ingredient is a bivariate local limit theorem, which follows from a general framework developed at the end of the paper.
   
   Local limit theorems happen to play a crucial role in the
treatment of questions from statistical mechanics, where they provide a rigorous justification of the \emph{equivalence of ensembles} principle. Twenty years ago, \citet{fristedt_structure_1993} introduced similar ideas to study the structure of uniformly drawn random partitions of
large integers. A few years later, \citet{baez-duarte_hardy-ramanujan_1997} applied a local limit theorem technique to derive a short proof of the Hardy-Ramanujan formula \eqref{eq:hardy-ramanujan}. The first implementation of these ideas in a two-dimensional context seems to be due to \citet{sinai_probabilistic_1994}, although the setting differs from ours. A general presentation of these techniques was discussed by \citet{vershik_statistical_1996}. In a recent paper, \citet{bogachev_universality_2011} presented among other results a detailed
proof of Sina{\u\i}'s approach. Let us mention that the strong anisotropy which is inherent in the problem that we address makes the implementation of this program more delicate.

\section{Notations and statement of the results}
\label{sec:notations}

Let $\Z_+ = \{0,1,2,3,\dots\}$ and $\N = \Z_+ \setminus\{0\}$ denote respectively the set of non-negative numbers and the set of positive integers.

We will use the standard Landau notations $a_n = o(b_n)$ or $a_n = \bigO(b_n)$ for sequences $(a_n)$ and $(b_n)$ satisfying respectively $\limsup \left|\frac{a_n}{b_n}\right| = 0$ or $\limsup \left|\frac{a_n}{b_n}\right| < +\infty$. Also, we will write $a_n \asymp b_n$ if $a_n$ and $b_n$ have the same order of magnitude as $n$ tends to infinity, that is to say if both $a_n = \bigO(b_n)$ and $b_n = \bigO(a_n)$ hold.

\begin{definition}
    Let $\X$ be a subset of $\Z_+^2$. For every $n \in \Z_+^2$, a \emph{partition} of $n$ with \emph{parts} in $X$ is a finite unordered family of elements of $\X$ whose sum is $n$. It can be represented by a \emph{multiplicity} function $\omega \colon X \to \Z_+$ such that $\sum_{x\in \X} \omega(x)\cdot x = n$. For $x \in \X$, we say that $\omega(x)$ is the multiplicity of the part $x$ in the partition. The partitions of $n$ with parts in $X$ constitute the set
    \[
        \Omega_\X(n) := \left\{\omega \in \Z_+^\X : \sum_{x\in\X} \omega(x)\cdot x = n\right\}.
    \]
    Finally, we write $p_X(n) := |\Omega_\X(n)|$ for the number of partitions of $n$ with parts in $\X$.
\end{definition}

Following the works of Wright and Robertson \cite{wright_partitions_1958,robertson_asymptotic_1960,robertson_partitions_1962}, we will focus on two particular sets of parts in this article, namely:
\begin{itemize}
    \item $X=\N^2$ which corresponds to partitions in which no part has a zero component,

    \item $X=\Z_+^2\setminus\{0\}$ which corresponds to the case of general partitions, in which parts may have a zero component. We still have to exclude the zero part in order to ensure that every vector has only finitely many partitions.
\end{itemize}

    The following theorem states the main results of the paper. It describes the asymptotic behaviour of $p_X(n)$ in the case of partitions without zero components as well as in the case of general partitions, outside Wright's region.  First, we need to introduce the following auxiliary functions of $\alpha > 0$:
\begin{gather*}
    \Phi(\alpha) = \sum_{r \geq 1} \frac{1}{r^2}\frac{e^{-\alpha r}}{1 - e^{-\alpha r}},\;
    \Theta(\alpha) = - \frac{\Phi'(\alpha)}{\sqrt{\Phi(\alpha)}},\;
    \overline{\Phi}(\alpha) = \Phi(\alpha) + \frac{\pi^2}{6},\;
    \overline{\Theta}(\alpha) = - \frac{\overline{\Phi}'(\alpha)}{\sqrt{\overline{\Phi}(\alpha)}},\\
    \Psi(\alpha) = \sum_{r \geq 1} \frac{1}{r} \frac{e^{-\alpha r}}{1-e^{-\alpha r}}, \;
    \Delta(\alpha) = 2\Phi(\alpha)\Phi''(\alpha) - \Phi'(\alpha)^2, \;
    \overline{\Delta}(\alpha) = 2\overline{\Phi}(\alpha)\overline{\Phi}''(\alpha) - \overline{\Phi}'(\alpha)^2.
\end{gather*}

Consider two sequences $(n_1(k))_k$ and $(n_2(k))_k$ of positive integers. In the sequel, the limits and asymptotic comparisons are to be understood as $k$ approaches infinity. The index $k$ will remain implicit.
\begin{theorem}
    \label{thm:main}
	Assume that both $n_1$ and $n_2$ tend to infinity under the conditions $n_1 = \bigO(\sqrt{n_2})$ and $\log (n_2) = o(n_1)$. 
    \begin{enumerate}[(ii)]
        \item If $\alpha_n > 0$ is the unique solution of $\Theta(\alpha_n)=\dfrac{n_1}{\sqrt{n_2}}$, then
            \[
                p_{\N^2}(n_1,n_2) \sim \frac{1}{2\pi} \frac{\Phi(\alpha_n)}{n_2}\frac{e^{-\frac{1}{2}\Psi(\alpha_n)}}{\sqrt{\Delta(\alpha_n)}}\exp\left\{\left(\alpha_n \Theta(\alpha_n)+2\sqrt{\Phi(\alpha_n)}\right)\sqrt{n_2}\right\}.
            \]
        \item If $\alpha_n > 0$ is the unique solution of $\overline{\Theta}(\alpha_n)=\dfrac{n_1}{\sqrt{n_2}}$, then
            \[
                p_{\Z_+^2\setminus\{0\}}(n_1,n_2) \sim \frac{1}{(2\pi )^{\frac{3}{2}}}\left(\frac{\overline{\Phi}(\alpha_n)}{n_2}\right)^{\frac{5}{4}}\frac{e^{\frac{1}{2}\Psi(\alpha_n)}}{\sqrt{\overline{\Delta}(\alpha_n)}}\exp\left\{\left(\alpha_n\overline{\Theta}(\alpha_n)+2\sqrt{\overline{\Phi}(\alpha_n)}\right)\sqrt{n_2}\right\}.
            \]
    \end{enumerate}
\end{theorem}

We will present a complete proof of (i) and will state along the proof the additional arguments which are needed for (ii).

Although the formul\ae{} in Theorem \ref{thm:main} involve an
implicit function $\alpha_n$ of $(n_1,n_2)$, remark that we can
actually derive explicit expansions in terms of $(n_1,n_2)$ when
$n_1$ is negligible compared to $\sqrt{n_2}$, which condition is equivalent to
$\alpha_n \to +\infty$.  Notice indeed that the auxiliary functions
$\Phi,\overline{\Phi},\Theta,\overline{\Theta},\Delta,\overline{\Delta}$,
and $\Psi$ admit simple asymptotic expansions in terms of the arithmetic
function $\sigma_2(m) = \sum_{d\mid m} d^2$ as $\alpha \to +\infty$.
This follows from the Lambert series \cite[Section 4.71]{titchmarsh_functions_1976} elementary formul\ae{}
\begin{equation}
    \label{eq:identities_phi}
\Phi(\alpha) = \sum_{m \geq 1}\frac{\sigma_2(m)}{m^2} e^{-\alpha m},\qquad  
-\Phi'(\alpha) = \sum_{m \geq 1} \frac{\sigma_2(m)}{m} e^{-\alpha m}.
\end{equation}
Asymptotic expansions of $\Theta^{-1},\overline{\Theta}{}^{-1}$ can be computed effectively from there by an iterative method or by using the Lagrange reversion formula.

Let us show how these simples ideas allows us to extend the previous results by Nanda and Robertson. For example, the following application of case (i) provides additional corrective terms in the expansion given by \citet[Theorem 2]{robertson_asymptotic_1960} which was stated for the special case $K=1$, that is to say $n_1 = o(n_2^{1/3})$.

\begin{corollary}
    \label{cor:2}
    There exists a sequence $(c_k)$ of rational numbers
     such that for all $K \in \N$, if $n_1$ and $n_2$ tend to $+\infty$ such that $n_1^{2K+1} = o(n_2^K)$ and $\log(n_2) = o(n_1)$,
    \[
        p_{\N^2}(n_1,n_2) \sim \frac{n_1}{n_2} \frac{n_2^{n_1}}{(n_1!)^2}\exp\left\{\sum_{k=1}^{K-1}c_k\, \frac{n_1^{2k+1}}{n_2^k}\right\}.
    \]
\end{corollary}

Notice that the proof of this formula, which is given below, can be directly translated into an effective algorithm. For instance, we give here the first terms of the sequence $(c_k)$ which have been computed with the help of the \textit{Sage} mathematical software \cite{sage}:
\[
    c_1 = \frac{5}{4} ,\; c_2 = -\frac{805}{288} ,\; c_3 = \frac{6731}{576} ,\; c_4 = -\frac{133046081}{2073600} ,\; c_5 = \frac{170097821}{414720}, \;\ldots
\]

In the same way, an application of case (ii) of our theorem leads to an extension of formula~\eqref{eq:auluck1} which was stated under the condition $n_1 = o(n_2^{1/4})$, or equivalently $K=1$, in the work of Nanda~\cite{nanda_bipartite_1957}.

\begin{corollary}
    \label{cor:3}
    There exists a sequence $(\overline c_k)$ of real numbers
    such that for all $K \in \N$, if $n_1$ and $n_2$ tend to $+\infty$ such that $n_1^{K+1} = o(n_2^{K/2})$ and $\log(n_2) = o(n_1)$,
    \[
        p_{\Z_+^2}(n_1,n_2) \sim \frac{1}{n_1!}\left(\frac{\sqrt{6n_2}}{\pi}\right)^{n_1}\frac{1}{4n_2\sqrt{3}}\exp\left\{\pi\sqrt{\frac{2n_2}{3}}\right\}
        \exp\left\{\sum_{k=1}^{K-1} \overline{c}_k\,\frac{n_1^{k+1}}{n_2^{k/2}}\right\}
    \]
\end{corollary}
An effective computation of the first terms of the sequence $(\overline c_k)$ gives, with $a = \sqrt{\zeta(2)}$,
\[
    \overline c_1 = \frac{5a}{4}-\frac{1}{4a}, \; \overline c_2 = \frac{5}{8} -\frac{145a^2}{72},\; \overline c_3 = 6a^3 - \frac{1385a}{576} + \frac{5}{32a} + \frac{1}{192a^3}, \; \ldots
\]

\begin{proof}[Proof of Corollary \ref{cor:2}]
    As noted above, the condition $n_1 = o(\sqrt{n_2})$ implies that $\alpha_n$ tends to $+\infty$ (see the proof of Proposition~\ref{prop:implicit_parameters} for details). Since the both of $\Phi(\alpha)$ and $\sqrt{\Delta(\alpha)}$ are equivalent to $e^{-\alpha}$ as $\alpha$ tends to $+\infty$, and $\Psi(\alpha)$ tends to $0$, the non-exponential factor in the formula for $p_{\N^2}(n_1,n_2)$ of Theorem~\ref{thm:main} reduces to $\frac{1}{2\pi n_2}$. Also, the identities \eqref{eq:identities_phi} for $\Phi(\alpha)$ and $-\Phi'(\alpha)$ show that we can now work with the formal power series
    \[
        f(z) = \sum_{m=1}^\infty \frac{\sigma_2(m)}{m^2}\,z^m, \quad
        g(z) = \frac{1}{f(z)}\left(-z\frac{d}{dz}f(z)\right)^2.
    \]
    Namely, the equation $\Theta(\alpha_n) = n_1/\sqrt{n_2}$ corresponds formally to $g(z_n)= n_1^2/n_2$ with $z_n := e^{-\alpha_n}$. Since $g(z)$ has no constant term, we obtain by reversion an infinite asymptotic expansion (we use here the Poincaré notation $\sim$ to denote an asymptotic series):
    \[
        e^{-\alpha_n} \sim \frac{n_1^2}{n_2} + \sum_{k=2}^{\infty} a_k\frac{n_1^{2k}}{n_2^k}.
    \]
    for some sequence of rationnals $(a_k)$ expressible with the help of the Lagrange inversion formula. From this point, it is now easy to derive the existence of two expansions (where $(b_k)$ and $(b'_k)$ are two sequences of rational numbers)
    \[
        \alpha_n \sim \log n_2 - 2 \log n_1 + \sum_{k=1}^\infty b_k \frac{n_1^{2k}}{n_2^{k}}, \quad
        \sqrt{\Phi(\alpha_n)} \sim \frac{n_1}{\sqrt{n_2}}\left[ 1 + \sum_{k=2}^\infty b'_k \frac{n_1^{2k}}{n_2^{k}}\right],
    \]
which, together with Theorem \ref{thm:main} and Stirling's formula, prove the result.
\end{proof}

In both Corollary~\ref{cor:2} and Corollary~\ref{cor:3}, notice that the boundary of the domain of validity for the expansions is actually asymptotic, as $K$ grows larger, to the critical regime $n_1 \asymp \sqrt{n_2}$ which is represented by the thick line in Figure~\ref{fig:regions_diagram}. In this critical case, Theorem \ref{thm:main} still applies but does not lead to any much simpler expression since, when $n_1/\sqrt{n_2}$ converges quickly enough to some positive constant, all terms depending on $\alpha_n$ tend to constant coefficients. Still, the theorem provides the existence and some expressions of the exponential rate functions
\[
    h(t) := \lim_{n\to +\infty} \frac{1}{\sqrt{n}} \log p_{\N^2}(\lfloor t\sqrt{n}\rfloor, n), \quad \overline h(t) =
\lim_{n\to +\infty} \frac{1}{\sqrt{n}} \log p_{\Z_+^2\setminus\{0\}}(\lfloor t\sqrt{n}\rfloor, n),
\]
which are defined for all $t > 0$. Figure~\ref{fig:curve} shows the graphs of these functions.

\begin{figure}[h!]
     \label{fig:curve}
     \centering
     \includegraphics[width=0.7\textwidth]{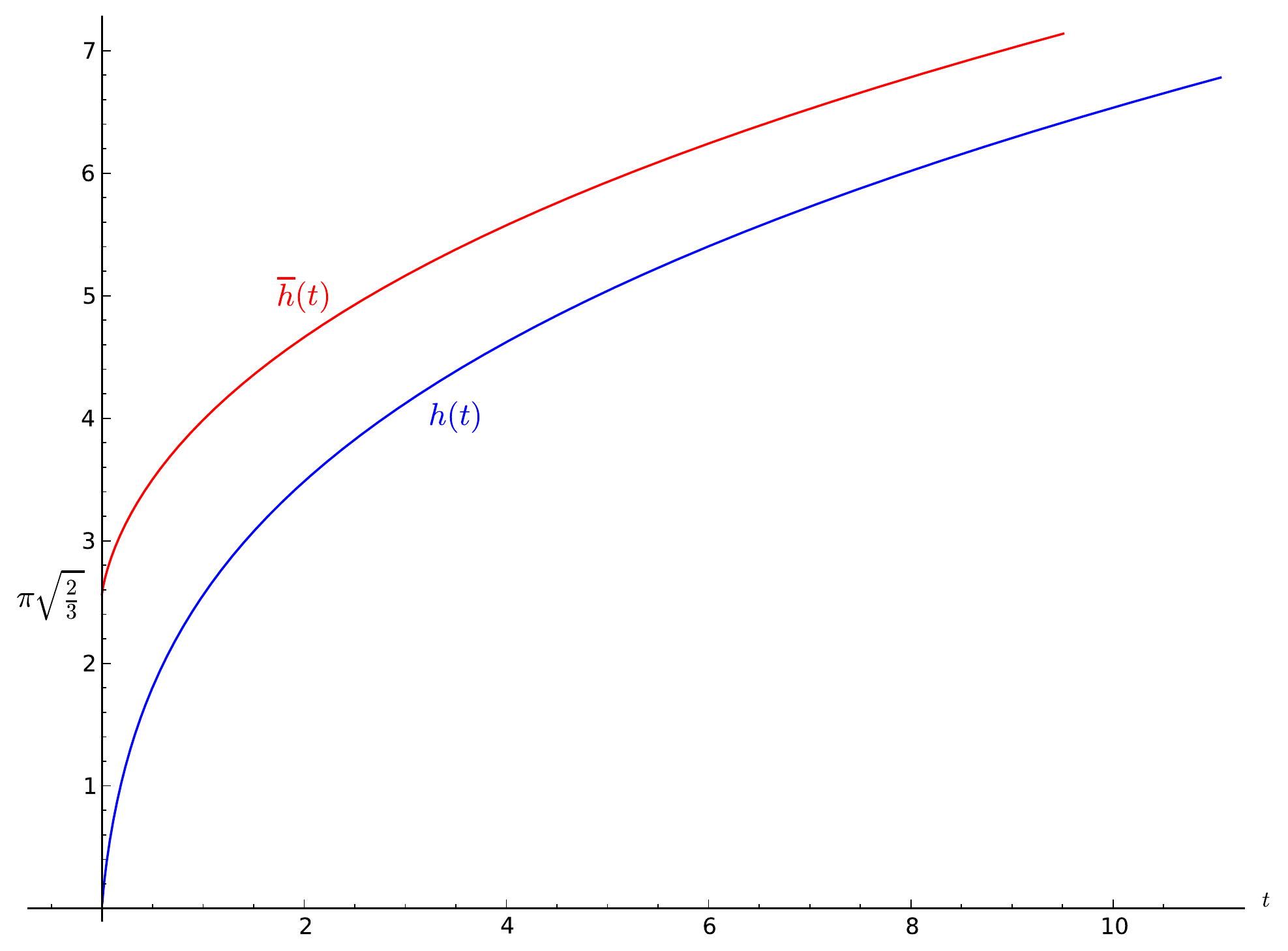}
     \caption{Comparison of the rate functions $h$ and $\overline{h}$ corresponding respectively to partitions without zero components and to partitions where zero components are allowed.}
\end{figure}

\section{The probabilistic model}
\label{sec:model}

In this section, we introduce a family of Gibbs probability measures on the set $\Omega_\X := \bigcup_{n\in \Z_+^2}\Omega_X(n)$ of all partitions, given some fixed set of parts $\X$. The idea is that, while the uniform distribution on the set $\Omega_\X(n)$ of partitions on $n$ is hard to describe, it is much easier to define a distribution on the larger space $\Omega_\X$, that will give the exact same weight to every partition of $n$. Let $\alpha,\beta \in (0,+\infty)$ be two \emph{shape
    parameters} to be chosen later and write $\lambda = (\alpha,\beta)$. To each choice of $\lambda$, we associate a probability measure $\PP_\lambda$ on the discrete space $\Omega_\X$ such that for every $\omega \in \Omega_\X$,
\begin{equation*}
    \PP_\lambda(\omega) := \frac{\exp\left\{-\sum_{x \in X} \omega(x) \langle \lambda, x\rangle\right\}}{\sum_{\omega \in \Omega_X}\exp\left\{-\sum_{x \in X} \omega(x) \langle \lambda, x\rangle\right\}}.
\end{equation*}
For each partition $\omega \in \Omega_X$, let us introduce the key quantity $N(\omega) := \sum_{x\in X} \omega(x)\cdot x$, which we will see as a random variable with values in $\Z_+^2$. By definition, we have $\omega \in \Omega_X(N(\omega))$ for every partition $\omega$. Furthermore, the probability $\PP_\lambda(\omega)$ becomes
\[
    \PP_\lambda(\omega) = \frac{1}{Z_\lambda}e^{-\sca{\lambda, N(\omega)}},\qquad\text{where } Z_\lambda := \sum_{\omega \in \Omega_\X} e^{-\sca{\lambda, N(\omega)}}.
\]
The normalization constant $Z_\lambda$ is usually referred to as the \emph{partition function} of the system in the statistical mechanics literature.

In accordance with the previous discussion, remark that the conditional distribution of $\PP_\lambda$ on $\Omega_X(n)$ yields the uniform measure since the probability of a partition $\omega$ only depends on $N(\omega)$. Moreover, since all partitions of $\Omega_X(n)$ are given equal weight $\frac{1}{Z_\lambda}e^{-\sca{\lambda,n}}$, the total weight of $\Omega_X(n)$ is equal to $\frac{1}{Z_\lambda}p_X(n)e^{-\sca{\lambda,n}}$ and therefore
\begin{equation}
    \label{eq:link}
    p_X(n) = \frac{Z_\lambda}{e^{-\sca{\lambda,n}}} \PP_\lambda(N=n).
\end{equation}
Probabilistic intuition dictates our strategy: calibrate the parameter $\lambda$ as a function of $n$
so that the distribution of the random vector $N$ concentrates around
$n$ under the probability measure $\PP_\lambda$. This way, we will
be able to ensure a polynomial decrease for the quantity
$\PP_\lambda(N=n)$. A natural choice
to enforce this behaviour is to take $\lambda_n = (\alpha_n,\beta_n)$
such that $\EE_\lambda(N)$ is close enough to $n$. This is achieved by
choosing the couple $(\alpha_n,\beta_n)$ defined by the equations~\eqref{eq:parameters}
of Section~\ref{sec:calibration}. Looking back to $p_X(n)$, we see that we
need to estimate precisely the partition function $Z_\lambda$ as well as
$\PP_\lambda(N=n)$. The former will be done by a careful approximation of
$\log Z_\lambda$ in section 4 and the latter will be deduced from estimates
of the first and second derivatives of $\log Z_\lambda$ together with
a Gaussian local limit theorem statement proven in section 5.

Before we turn to more technical discussions, let us remark that under
the probability measure $\PP_\lambda$, the random variables $\omega(x)$
for $x \in \X$ are mutually independent and that their distribution is geometric. More precisely, we have for all $k \in \Z_+$,
\[
    \PP_\lambda(\omega(x) = k) = e^{-k\sca{\lambda,x}} \,(1-e^{-\sca{\lambda,x}}).
\]
Finally, a fruitful consequence of the independence in this model is the fact that the partition function $Z_\lambda$ can be written as an infinite product: 
\begin{equation}
    \label{eq:product-formula}
    Z_\lambda = \prod_{x \in \X} \frac{1}{1-e^{-\sca{\lambda,x}}}.
\end{equation}
Let us mention that~\eqref{eq:product-formula} is the bipartite partition analogue of the famous Euler product formula for the usual partition generating function.

\section{Approximation of the logarithmic partition function}
\label{sec:approximation}

Because of the product formula \eqref{eq:product-formula}, the logarithm of the partition function $Z_\lambda$ can be expressed as the sum of an absolutely convergent series:
\[
\log Z_\lambda = -\sum_{x\in \X} \log(1-e^{-\sca{\lambda,x}})
 = \sum_{x\in\X} \sum_{r= 1}^\infty \frac{e^{-r\sca{\lambda,x}}}{r}.
\]
Let us recall that we consider the case $\X=\N^2$ of partitions whose parts have non-zero components. The logarithmic partition function thus writes
\begin{equation}
    \label{eq:logZ-intermediate}
    \log Z_\lambda = \sum_{x_1\geq 1}\sum_{x_2\geq 1} \sum_{r \geq 1} \frac{e^{-\alpha x_1 r}}{r}e^{-\beta x_2 r}.
\end{equation}
Let $\zeta$ and $\Gamma$ denote respectively the Riemann zeta function and the Euler gamma function \citep{titchmarsh_functions_1976}. Also consider for every $\alpha > 0$ and $s \in \CC$ the Dirichlet series $D_\alpha(s)$ defined by
\[
    D_\alpha(s) := \sum_{k \geq 1}\sum_{r \geq 1} \frac{e^{-\alpha k r}}{r^s} = \sum_{r \geq 1} \frac{1}{r^s}\frac{e^{-\alpha r}}{1-e^{-\alpha r}}.
\]
Recalling that the Cahen-Mellin inversion formula~\cite{titchmarsh_functions_1976} yields for every $c > 0$ and $t > 0$,
\[
    e^{-t} = \frac{1}{2i\pi} \int_{c-i\infty}^{c+i\infty} \Gamma(s)t^{-s}\,ds,
\]
we can rewrite the identity \eqref{eq:logZ-intermediate} for every $c > 1$ as
\begin{align*}
    \log Z_\lambda & = \frac{1}{2i\pi} \sum_{x_1 = 1}^\infty \sum_{x_2 = 1}^\infty\sum_{r = 1}^\infty \int_{c-i\infty}^{c+i\infty} \frac{e^{-r\alpha x_1}}{r} \Gamma(s)(r\beta x_2)^{-s}\,ds \\
                   & = \frac{1}{2i\pi}\int_{c-i\infty}^{c+i\infty} \Gamma(s)\zeta(s)D_\alpha(s+1)\frac{ds}{\beta^s}, 
\end{align*}
the exchange in the order of summation being justified by the Fubini theorem. We proved that the logarithmic partition function admits an integral representation of the form
\[
\log Z_\lambda = \frac{1}{2i\pi}\int_{c-i\infty}^{c+i\infty} M_\alpha(s)\frac{ds}{\beta^s},
\]
where $M_\alpha(s) := \Gamma(s)\zeta(s)D_\alpha(s+1)$. We will now see how to recover from the residues of the meromorphic function $M_\alpha$ the asymptotic behaviour of $\log Z_\lambda$ and its derivatives when $\beta$ tends towards $0$.

\begin{proposition}
    \label{prop:approximation}
    For every non-negative integers $m,p,q$, there exists a decreasing function $C_m^{p,q}(\alpha)$ of $\alpha > 0$ with a positive limit as $\alpha \to \infty$, such that the remainder function $R_m(\alpha,\beta)$ defined by
    \[
        R_m(\alpha,\beta) := \log Z_{(\alpha,\beta)} - \frac{D_\alpha(2)}{\beta} - \sum_{k=0}^m \frac{(-1)^{k}\zeta(-k)D_\alpha(1-k)}{k!} \beta^k
    \]
    satisfies $|\partial_ {\vphantom{\beta}\alpha}^{p} \partial_\beta^{\vphantom{p}q} R_{m}(\alpha,\beta)| \leq C_m^{p,q}(\alpha) \,e^{-\alpha}\beta^{m-q+\frac{1}{2}}$ for all $\alpha,\beta > 0$.
\end{proposition}

\begin{proof}
    Let us recall that the Riemann function $\zeta(s)$ is meromorphic on $\CC$ and that it has a unique pole at $s=1$, at which the residue is $1$. The Euler gamma function $\Gamma(s)$ is also meromorphic on $\CC$ and has poles at every integer $k \leq 0$.

Let $m$ be a positive integer and $\gamma_m = -m-\frac{1}{2}$. We are going to apply Cauchy's residue theorem to the meromorphic function $M_\alpha$ with the rectangular contour $C$ defined by the segments $[2-iT,2+iT]$, $[2+iT,\gamma_m+iT]$, $[\gamma_m+iT,\gamma_m-iT]$ and $[\gamma_m-iT,2-iT]$, where $T > 0$ is some positive real we will let go to infinity. Computation of the residues of $M_\alpha$ in this stripe yields
\[
\frac{1}{2i\pi}\int_{C^+}M_\alpha(s)\frac{ds}{\beta^s} = \frac{D_\alpha(2)}{\beta} + \sum_{k=0}^m \frac{(-1)^k\zeta(-k)D_\alpha(1-k)}{k!}\beta^k.
\]
In order to prove that the contributions of the horizontal segments in the left-hand side integral vanish for $T \to \infty$, we use the three following facts:
\begin{enumerate}[(iii)]
    \item From the complex version of Stirling's formula, we know that $|\Gamma(\sigma+i\tau)|$ decreases exponentially fast when $|\tau|$ tends to $+\infty$, uniformly in every bounded stripe \citep[p. 151]{titchmarsh_functions_1976}.
    \item Also $|\zeta(\sigma + i\tau)|$ is polynomially bounded in $|\tau|$ as $|\tau| \to +\infty$, uniformly in every bounded stripe \citep[p. 95]{titchmarsh_riemann_1986}.
    \item Finally, note that for all $\sigma_0 \in \RR$,
        \[
            \sup_{\sigma \geq \sigma_0} \left|D_\alpha(\sigma + i\tau)\right| \leq D_\alpha(\sigma_0) = \sum_{r\geq 1} \frac{1}{r^{\sigma_0}}\frac{e^{-\alpha r}}{1-e^{-\alpha r}} < \infty.
        \]
\end{enumerate}
From these observations, we see that the integrals along horizontal lines vanish as $T \to \infty$:
\[
	\lim_{T\to+\infty} \int_{2+iT}^{\gamma_m+iT} M_\alpha(s) \dfrac{ds}{\beta^s} = \lim_{T\to+\infty} \int_{\gamma_m-iT}^{2-iT} M_\alpha(s) \dfrac{ds}{\beta^s} = 0.
\]
Furthermore, $M_\alpha(s)\beta^{-s}$ is integrable on the vertical line $(\gamma_m-i\infty,\gamma_m+i\infty)$, so that in the limit $T \to \infty$, we obtain
\[
	\log Z_\lambda = \frac{D_\alpha(2)}{\beta}+\sum_{k=0}^m \frac{(-1)^k\zeta(-k)D_\alpha(1-k)}{k!}\beta^k + \frac{1}{2i\pi}\int_{\gamma_m-i\infty}^{\gamma_m+i\infty} M_\alpha(s)\frac{ds}{\beta^s}.
\]
Thus, the remainder function $R_m(\alpha,\beta)$ is actually equal to the integral term in the right-hand side. We need to control its derivatives. Note that the derivatives
\[
    \partial_{\vphantom{\beta}\alpha}^p \partial_\beta^q \;\frac{M_\alpha(s)}{\beta^s} = (-1)^q \frac{\Gamma(s+q)}{\beta^{s+q}} \zeta(s)\partial_\alpha^p D_\alpha(s+1)
\]
are well defined and integrable on the vertical line $(\gamma_m-i\infty,\gamma_m+i\infty)$ thanks again to the facts (i) and (ii), as well as the analogue of (iii) for the $\partial_\alpha^p$ derivative of $D_\alpha$. It is easy to check that the result follows with
\[
    C_m^{p,q}(\alpha) \,e^{-\alpha} = \frac{|\partial_\alpha^p D_\alpha(\gamma_m+1)|}{2\pi} \int_{-\infty}^\infty |\zeta(\gamma_m+i\tau)||\Gamma(\gamma_m+q+i\tau)| \,d\tau.\qedhere
\]
\end{proof}

\begin{remark}
    In the case $X = \Z_+^2\setminus\{0\}$, the logarithmic partition function becomes $\log Z_\lambda + \Psi(\alpha) + \Psi(\beta)$, where $\Psi(\cdot) = D_\cdot(1)$ is the function defined in Section~\ref{sec:notations}. The two additional terms correspond respectively to horizontal and vertical one-dimensional integer partitions. In that case, the expansion as $\beta \to 0^+$ involves the expansion of $\Psi(\beta)$ (which was the basis of
    \citep{hardy_asymptotic_1918}) and can be written informally as
    \[
        \frac{\Phi(\alpha)+\zeta(2)}{\beta} + \frac{1}{2}\log \beta + \frac{\Psi(\alpha)}{2} - \frac{1}{2}\log (2\pi) + \left(\frac{D_\alpha(0)}{12}-\frac{1}{24}\right)\beta + \dots
    \]
\end{remark}

\section{Calibration of the shape parameters}
\label{sec:calibration}

In this section, we find appropriate values for the parameters $\lambda = (\alpha,\beta)$ as functions of $n$ for which $\EE_\lambda(N)$ is asymptotically close to $n$.
Since the distribution of the random vector $N$ under $\PP_\lambda$ is given by a Gibbs measure,
\[
    \EE_\lambda(N) = - \mathrm{grad}\,(\log Z_\lambda) = 
    - \begin{bmatrix}
        \partial_\alpha \log Z_\lambda\\
        \partial_\beta \log Z_\lambda
    \end{bmatrix}.
\]
Let us recall that by definition, the function $\Phi$ introduced in Section~\ref{sec:notations} is
\[
    \forall \alpha > 0,\qquad \Phi(\alpha) := \sum_{r\geq 1} \frac{1}{r^2}\frac{e^{-\alpha r}}{1-e^{-\alpha r}},
\]
which is exactly $D_\alpha(2)$ for the Dirichlet series $D_\alpha(s)$ introduced in Section~\ref{sec:approximation}. The approximation given in Proposition~\ref{prop:approximation} applied with $(m,p,q) = (1,1,0)$ and $(m,p,q)=(1,0,1)$ yields the existence of two decreasing functions $C_1$ and $C_2$ with finite limits as $\alpha \to +\infty$ such that
\[
    \left| \partial_\alpha \log Z_\lambda - \frac{\Phi'(\alpha)}{\beta}\right| \leq e^{-\alpha}C_1(\alpha)\qquad\text{and} \qquad \left|\partial_\beta \log Z_\lambda +\frac{\Phi(\alpha)}{\beta^2} \right| \leq e^{-\alpha} C_2(\alpha).
\]
This justifies the choice made in the next proposition to define the shape parameters $\alpha_n$ and $\beta_n$ through the implicit equations~\eqref{eq:parameters}. We start with a lemma.

\begin{lemma}
    \label{lem:log-convexity}
    The function $\Phi$ is logarithmically convex on $(0;+\infty)$. 
\end{lemma}

\begin{proof}
    The function $\log \Phi$ being smooth, its convexity is equivalent to the inequality
    \[
        \frac{d^2}{d\alpha^2} (\log \Phi(\alpha)) = \frac{\Phi(\alpha)\Phi''(\alpha)-\Phi'(\alpha)^2}{\Phi(\alpha)^2} \geq 0,
    \]
    which follows immediately from the Cauchy-Schwarz inequality since
    \[
        \Phi(\alpha) = \sum_{k \geq 1}\sum_{r \geq 1} \frac{e^{-\alpha k r}}{r^2},\qquad \Phi'(\alpha) = -\sum_{k\geq 1}\sum_{r \geq 1} \frac{k e^{-\alpha k r}}{r}\quad\text{and}\quad \Phi''(\alpha) = \sum_{k\geq 1}\sum_{r \geq 1} k^2 e^{-\alpha k r}.\qedhere
    \]
\end{proof}

\begin{proposition}
    \label{prop:implicit_parameters}
    For all $n = (n_1,n_2) \in \N^2$, there exists a unique couple $(\alpha_n,\beta_n) \in (0;+\infty)^2$ such that
    \begin{equation}
        \label{eq:parameters}
            \dfrac{-\Phi'(\alpha_n)}{\beta_n} = n_1\qquad\text{and}\qquad
            \dfrac{\Phi(\alpha_n)}{\beta_n^2} = n_2.
    \end{equation}
\end{proposition}

\begin{proof}
    Eliminating $\beta_n$ in the implicit equations~\eqref{eq:parameters}, we need only prove the existence of a unique $\alpha_n > 0$ such that
    \[
        \frac{1}{2}\frac{\Phi'(\alpha_n)}{\sqrt{\Phi(\alpha_n)}} = -\frac{1}{2}\frac{n_1}{\sqrt{n_2}}.
    \]
    We recognize the derivative of the function $\sqrt{\Phi(\alpha)} = \exp\left(\frac{1}{2}\log \Phi(\alpha)\right)$ which is strictly convex by Lemma~\ref{lem:log-convexity}, so that $\Phi'/\sqrt{\Phi}$ is continuously increasing. In addition, it is easy to check that
    \[
        \lim_{\alpha \to 0^+} \frac{\Phi'(\alpha)}{\sqrt{\Phi(\alpha)}} = \lim_{\alpha \to 0^+} \frac{-\zeta(3)/\alpha^2}{\sqrt{\zeta(3)/\alpha}} = -\infty 
        \quad\text{and}\quad
        \lim_{\alpha \to +\infty} \frac{\Phi'(\alpha)}{\sqrt{\Phi(\alpha)}} = \lim_{\alpha\to+\infty} \frac{-e^{-\alpha}}{\sqrt{e^{-\alpha}}} = 0.
    \]
    This concludes the proof.
\end{proof}

Recall that the assumptions of Theorem~\ref{thm:main} are $n_1 \to +\infty$ and $n_2 \to +\infty$ with the conditions $n_1 = \bigO(\sqrt{n_2})$ and $\log(n_2)=o(n_1)$. Consider the couple of parameters $\alpha_n$ and $\beta_n$ defined by the equations~\eqref{eq:parameters}. A consequence of the proof above is that, under these assumptions, the sequence $\alpha_n$ is bounded away from $0$ (but not from $+\infty$). In addition,
\begin{equation}
    \label{eq:limit-parameters}
    \frac{e^{-\alpha_n}}{\beta_n} \asymp n_1
    ,\qquad \frac{e^{-\alpha_n}}{\beta_n^2} \asymp n_2 
    , \qquad\text{and}\qquad \beta_n \asymp \frac{n_1}{n_2}.
\end{equation}

\section{The local limit theorem and its application}

In the sequel, the parameters $\lambda_n = (\alpha_n,\beta_n)$ are
chosen according to the equations~\eqref{eq:parameters}. The aim of
the present section is to show that the random vector $N$ satisfies a
local limit theorem. Sufficient conditions for such a theorem to hold
are given in Proposition~\ref{prop:framework} of the next section. We check
that these conditions are satisfied for our model. Finally, an application of Proposition~\ref{prop:framework} leads to the proof of Theorem~\ref{thm:main}.

\subsection{An estimate of the covariance matrix}

The assumptions of Proposition~\ref{prop:framework} require a good estimate of the covariance matrix $\Gamma_{\lambda_n}$ of the random vector $N$ under the measure $\PP_{\lambda_n}$.
Since we have a Boltzmann-type model with a Gibbs measure $\PP_\lambda$, the covariance matrix $\Gamma_\lambda$ of $N$ is simply given by the second derivatives of the logarithmic partition function,
\[
    \Gamma_\lambda = \Hess (\log Z_\lambda) =
        \EE_\lambda\left(\begin{bmatrix}
            \partial_\alpha^2 \log Z_\lambda & \partial_\alpha\partial_\beta \log Z_\lambda\\
            \partial_\alpha\partial_\beta \log Z_\lambda & \partial_\beta^2 \log Z_\lambda
	\end{bmatrix} \right).
\]
Denoting by $\Sigma_\lambda$ the symmetric matrix
\[
	\Sigma_\lambda  := \begin{bmatrix}
         \dfrac{\Phi''(\alpha)}{\beta} & -\dfrac{\Phi'(\alpha)}{\beta^2}\\
         -\dfrac{\Phi'(\alpha)}{\beta^2} & \dfrac{2\Phi(\alpha)}{\beta^3}\\
     \end{bmatrix},
\]
an application of Proposition \ref{prop:approximation} with $m=2$ and $p+q=2$ yields the existence of a positive decreasing function $C(\alpha)$ with a positive limit as $\alpha\to\infty$ such that
\begin{equation}
    \label{eq:error_gamma}
    \left\| \Gamma_\lambda - \Sigma_\lambda \right\| \leq e^{-\alpha}C(\alpha).
\end{equation}

We can now state two crucial consequences of this approximation concerning $\Gamma_{\lambda_n}$. The first one concerns the precise asymptotic behaviour of its determinant while the second one shows that its eigenvalues go to $+\infty$. Let us recall that the function $\Delta(\alpha)$ introduced in Section~\ref{sec:notations} is defined by
\[
    \Delta(\alpha) := 2\Phi(\alpha)\Phi''(\alpha)-\Phi'(\alpha)^2 = 
	 \det \begin{bmatrix}
		\Phi''(\alpha) & -\Phi'(\alpha)\\
		- \Phi'(\alpha) & 2 \Phi(\alpha)
	\end{bmatrix}.
\]

\begin{proposition}
    \label{prop:determinant}
    Under the assumptions of Theorem~\ref{thm:main}, $\det \Gamma_{\lambda_n} \sim \dfrac{\Delta(\alpha_n)}{\beta_n^4} \asymp (n_2)^2$.
\end{proposition}

\begin{proof}
    Using the approximation~\eqref{eq:error_gamma} and the fact that for all integers $p \geq 0$, one has $\Phi^{(p)}(\alpha) = (-1)^p e^{-\alpha} + \bigO(e^{-2\alpha})$ as $\alpha \to +\infty$, we need to prove that $\Delta(\alpha)$ does not vanish for $\alpha > 0$ and that $e^{-2\alpha_n}/\beta_n^3$ is negligible compared to $\Delta(\alpha_n)/\beta_n^4$. By Lemma~ \ref{lem:log-convexity}, $\Phi$ is logarithmically convex, so that we have $\Phi'(\alpha)^2 \leq
    \Phi''(\alpha)\Phi(\alpha)$. As a consequence, $\Delta(\alpha) > 0$ for all $\alpha > 0$. In addition, the estimates~\eqref{eq:limit-parameters} imply that
    \[
        \frac{\Delta(\alpha_n)}{\beta_n^4} \asymp \frac{e^{-2\alpha_n}}{\beta_n^4} \asymp (n_2)^2 \qquad\text{and}\qquad \frac{e^{-2\alpha_n}}{\beta_n^3} \asymp n_1 n_2,
    \]
    which is enough to conclude because $n_1 = O(\sqrt{n_2})$ is negligible compared to $n_2$.
\end{proof}

Let us denote  by $\Gamma_\lambda(x,x) = \inp{x, \Gamma_\lambda x}$ for $x\in \RR^2$ the quadratic form on $\RR^2$ induced by the symmetric positive-definite matrix $\Gamma_\lambda$.

\begin{proposition}
    \label{prop:bounds_Gamma}
    Under the assumptions Theorem~\ref{thm:main}, there exist positive constants $C_1,C_{-1}$ such that for all $x \in \RR^2$,
    \begin{gather}
        \label{eq:bound_Gamma}
        \Gamma_{\lambda_n}(x,x) \leq C_1\left( n_1 |x_1|^2 + \frac{(n_2)^2}{n_1} |x_2|^2\right),\\
        \label{eq:bound_Gamma_inv}
        \Gamma_{\lambda_n}^{-1}(x,x) \leq C_{-1}\left(\frac{1}{n_1} |x_1|^2 + \frac{n_1}{(n_2)^2} |x_2|^2\right).
\end{gather}
\end{proposition}

\begin{proof}
    Let us first prove that the inequality~\eqref{eq:bound_Gamma} holds for $\Sigma_{\lambda_n}$ instead of $\Gamma_{\lambda_n}$.
	For every $x=(x_1,x_2) \in \RR^2$, the log-convexity of $\Phi$ (Lemma~\ref{lem:log-convexity}) $|\Phi'(\alpha)|^2 \leq \Phi''(\alpha)\Phi(\alpha)$ 
	and the inequality between arithmetic mean and geometric mean yields
	\[
		\left|\frac{\Phi'(\alpha)}{\beta^2}x_1x_2\right| \leq
		\sqrt{\frac{\Phi''(\alpha)}{\sqrt{2}\beta}|x_1|^2}\sqrt{\frac{\sqrt{2}\Phi(\alpha)}{\beta^3}|x_2|^2}
		\leq \frac{1}{2\sqrt{2}} \frac{\Phi''(\alpha)}{\beta}|x_1|^2 +
		\frac{\sqrt{2}}{2}\frac{\Phi(\alpha)}{\beta^3}|x_2|^2,
	\]
    which implies that for the positive constants $C_{\pm} = 1 \pm \frac{1}{\sqrt{2}}$ and for all $x \in \RR^2$,
    \[
        C_- \left[\frac{\Phi''(\alpha)}{\beta}|x_1|^2 + 
         \frac{2\Phi(\alpha)}{\beta^3}|x_2|^2 \right] \leq \Sigma_\lambda(x,x) \leq 
        C_+ \left[\frac{\Phi''(\alpha)}{\beta}|x_1|^2 + 
         \frac{2\Phi(\alpha)}{\beta^3}|x_2|^2 \right].
    \]
    In other words, the following matrix inequality holds for the L\"owner ordering $\loewleq$ (let us recall that two symmetric real matrices $A$ and $B$ satisfy $A \loewleq B$ if $B - A$ is positive semi-definite):
     \[
	C_- 	     	
     \begin{bmatrix}
         \dfrac{\Phi''(\alpha)}{\beta} & 0\\
         0 & \dfrac{2\Phi(\alpha)}{\beta^3}\\
     \end{bmatrix}
     \loewleq
     \Sigma_\lambda
     \loewleq C_+
     \begin{bmatrix}
         \dfrac{\Phi''(\alpha)}{\beta} & 0\\
         0 & \dfrac{2\Phi(\alpha)}{\beta^3}\\
     \end{bmatrix}.
     \]
     Considering the right-hand side of this inequality, and remembering that as a consequence of~\eqref{eq:limit-parameters}, we have $\Phi''(\alpha_n)/\beta_n \asymp n_1$ and $\Phi(\alpha_n)/\beta_n^3 \asymp (n_2)^2/n_1$, we see that the analogue of the bound \eqref{eq:bound_Gamma} for the quadratic form $\Sigma_{\lambda_n}$ holds. In order to complete the proof of the proposition, we need to control the error made when we replace $\Gamma_{\lambda_n}$ by $\Sigma_{\lambda_n}$. Using \eqref{eq:error_gamma}, for all $x \in \RR^2$,
    \[
        \left|\Gamma_{\lambda_n}(x,x) - \Sigma_{\lambda_n}(x,x)\right| \leq \|\Gamma_{\lambda_n} - \Sigma_{\lambda_n}\| \cdot \|x\|^2 \leq e^{-\alpha_n}C(\alpha_n)\|x\|^2.
    \]
    Under the conditions of Theorem~\ref{thm:main}, since $C(\alpha_n) \asymp 1$ and $e^{-\alpha_n} \asymp (n_1)^2/n_2$ which is negligible compared to both $n_1$ and $(n_2)^2/n_1$, we obtain
     \begin{equation}
         \label{eq:loewner_gamma}
         \frac{C_-}{2} 	     	
     \begin{bmatrix}
         n_1 & 0\\
         0 & \dfrac{(n_2)^2}{n_1}\\
     \end{bmatrix}
     \loewleq
     \Gamma_{\lambda_n}
     \loewleq 2C_+
     \begin{bmatrix}
         n_1 & 0\\
         0 & \dfrac{(n_2)^2}{n_1}\\
     \end{bmatrix},
 \end{equation}
     which in turn implies, using the decreasing property of matrix inversion with respect to L\"owner ordering,
     \begin{equation}
         \label{eq:loewner_gamma_inv}
	     \frac{1}{2C_+} 	     	
     \begin{bmatrix}
         \dfrac{1}{n_1} & 0\\
         0 & \dfrac{n_1}{(n_2)^2}\\
     \end{bmatrix}
     \loewleq \Gamma_{\lambda_n}^{-1}
     \loewleq \frac{2}{C_-}
     \begin{bmatrix}
         \dfrac{1}{n_1} & 0\\
         0 & \dfrac{n_1}{(n_2)^2}\\
     \end{bmatrix}.
     \end{equation}
     The right-hand sides of \eqref{eq:loewner_gamma} and \eqref{eq:loewner_gamma_inv} provide respectively \eqref{eq:bound_Gamma}
      and \eqref{eq:bound_Gamma_inv}.
\end{proof}

\begin{corollary}
    \label{cor:eigenvalue}
    Let $\sigma^2_n$ be the smallest eigenvalue of $\Gamma_{\lambda_n}$. It satisfies $\sigma_n^2 \asymp n_1$.
\end{corollary}

\begin{proof}
    This is an immediate consequence of the inequalities~\eqref{eq:loewner_gamma}.
\end{proof}

\subsection{The condition on the Lyapunov ratio}
\label{subsec:lyapunov}
We now check the second assumption of Proposition~\ref{prop:framework} below.
Let $\Gamma_\lambda^{1/2}$ be the uniquely defined symmetric positive-definite square root of $\Gamma_\lambda$. We introduce the following analogue of the scale-independent Lyapunov ratio \citep[p. 59]{bhattacharya_normal_2010}:
\[
    L_\lambda := \sup_{t \in \RR^d} \frac{1}{{\|\Gamma_\lambda^{1/2}t\|}^3}\sum_{x\in X} \EE_\lambda\left|\big\langle t,[\omega(x) - \EE_\lambda \omega(x)]x\big\rangle\right|^3.
\]
\begin{proposition}
    \label{prop:lyapunov-condition}
    Under the assumptions of Theorem~\ref{thm:main}, $L_{\lambda_n} = \bigO\left(\dfrac{1}{\sqrt{n_1}}\right)$.
\end{proposition}

\begin{proof}
For all $x \in \X$, let $\ovomega(x) := \omega(x) - \EE_\lambda\omega(x)$. Using the fact that $\Gamma_\lambda^{1/2}$ is symmetric and the Cauchy-Schwarz inequality, notice that we have for all $t\in \RR^2$,
\begin{align*}
    \sum_{x\in \X} \EE_\lambda\left|\inp{t,\ovomega(x)\cdot x}\right|^3 &= \sum_{x\in \X}
    \EE_\lambda\left|\inp{\Gamma_\lambda^{1/2}t,\ovomega(x)\cdot \Gamma_\lambda^{-1/2}x}\right|^3\\
    & \leq \|\Gamma_\lambda^{1/2}t\|^3\sum_{x\in \X}\|\Gamma_\lambda^{-1/2}x\|^3 \EE_\lambda|\ovomega(x)|^3.
\end{align*}
The bound~\eqref{eq:bound_Gamma_inv} of Proposition~\ref{prop:bounds_Gamma} and Jensen's inequality for the convex function $u \mapsto u^{3/2}$ entail the existence of some constant $C > 0$ such that for all $x \in \RR^2$,
\[
    \|\Gamma_{\lambda_n}^{-1/2}x\|^3 = \Gamma_{\lambda_n}^{-1}(x,x)^{3/2} \leq C\left[  \left(\frac{1}{n_1}\right)^{3/2}|x_1|^3 + \left(\frac{n_1}{(n_2)^2}\right)^{3/2}|x_2|^3\right].
\]  
Considering these two facts, we see that it will be enough to show the existence of two positive constants $C_1$ and $C_2$ such that
    \begin{equation}
        \label{eq:lyapunov_split}
        \sum_{x\in \X} |x_1|^3 \EE_\lambda(|\ovomega(x)|^3) \leq C_1 n_1\qquad\text{and}\qquad 
        \sum_{x\in \X} |x_2|^3 \EE_\lambda(|\ovomega(x)|^3) \leq C_2 \frac{(n_2)^3}{(n_1)^2}.
    \end{equation}
    In order to bound the third absolute moment $\EE_\lambda(|\ovomega(x)|^3)$, we first compute the fourth moment
    \[
        \EE_\lambda(|\ovomega(x)|^4) =
        \frac{e^{-\sca{\lambda,x}}(1+7e^{-\sca{\lambda,x}}+e^{-2\sca{\lambda,x}})}{(1-e^{-\sca{\lambda,x}})^4}
        \leq \frac{9e^{-\sca{\lambda,x}}}{(1-e^{-\sca{\lambda,x}})^4}.
    \]
    Reminding that $\EE_\lambda(|\ovomega(x)|^2) = e^{-\sca{\lambda,x}}/(1-e^{-\sca{\lambda,x}})^2$, and using the Cauchy-Schwarz inequality,
    \[
    \EE_\lambda(|\ovomega(x)|^3) \leq \sqrt{ \EE_\lambda(|\ovomega(x)|^2) \EE_\lambda(|\ovomega(x)|^4)} \leq \frac{3e^{-\sca{\lambda,x}}}{(1-e^{-\sca{\lambda,x}})^3}.
\]
For the first bound of~\eqref{eq:lyapunov_split}, we can thus write
    \[
        \sum_{x\in \X} |x_1|^3 \EE_\lambda(|\ovomega(x)|^3) \leq 3\sum_{x\in
            \X} \frac{|x_1|^3
            e^{-\sca{\lambda,x}}}{(1-e^{-\sca{\lambda,x}})^3} \leq \frac{3}{(1-e^{-\alpha})^3} \sum_{x_1\geq 1} |x_1|^3 e^{-\alpha x_1}\sum_{x_2 \geq 1} e^{-\beta x_2}.
    \]
    By \eqref{eq:limit-parameters}, we have $(1-e^{-\alpha_n}) \asymp 1$ and
    \[
        \sum_{x_1 \geq 1} |x_1|^3e^{-\alpha_n x_1} \asymp e^{-\alpha_n} \asymp \frac{(n_1)^2}{n_2},\qquad \sum_{x_2 \geq 1} e^{-\beta_n x_2} \asymp \frac{1}{\beta_n} \asymp \frac{n_2}{n_1}.
    \]
    Therefore, the first part of \eqref{eq:lyapunov_split} follows for some positive constant $C_1$. The second part is obtained similarly from
    \[
        \sum_{x_1 \geq 1} e^{-\alpha_n x_1} \asymp e^{-\alpha_n} \asymp \frac{(n_1)^2}{n_2} \quad\text{and}\quad \sum_{x_2 \geq 1}|x_2|^3 \frac{e^{-\beta_n x_2}}{(1-e^{-\beta_n x_2})^3} \asymp \frac{1}{\beta_n^4} \asymp \left(\frac{n_2}{n_1}\right)^4.\qedhere
    \]
\end{proof}

\subsection{The decrease condition on the characteristic function}

We finally check that the last condition of Proposition \ref{prop:framework} is satisfied. Consider the ellipse $\Ecal_\lambda$ defined by
\[
    \Ecal_\lambda := \left\{t \in \RR^2: \|\Gamma_\lambda^{1/2} t\| \leq \dfrac{1}{4L_\lambda}\right\} = \Gamma_\lambda^{-1/2}\left(\left\{u \in \RR^2 : \|u\| \leq \frac{1}{4L_\lambda}\right\}\right),
\]
where $L_\lambda$ is the Lyapunov ratio previously defined in Subsection~\ref{subsec:lyapunov}.

\begin{proposition}
    \label{prop:cramer-condition}
    Under the assumptions of Theorem~\ref{thm:main},
    \[
        \sup_{t \in [-\pi,\pi]^2 \setminus \Ecal_{\lambda_n}} \left| \EE_{\lambda_n}(e^{i\sca{t, N}})\right| = \bigO\left(\frac{1}{n_2\sqrt{n_1}}\right).
    \]

\end{proposition}

\begin{proof}
    Let us write $\varphi_\lambda(t) = \EE_\lambda(e^{i\sca{t,N}})$ for $t\in \RR^2$, the characteristic function of $N$. Observe that the following elementary inequality holds for all complex number $z$ with modulus $|z| < 1$:
    \begin{equation}
        \label{eq:elem_ineq}
        \left|\frac{1-|z|}{1-z}\right| \leq \exp\left\{\Re(z) - |z|\right\}.
    \end{equation}
    Applying \eqref{eq:elem_ineq} with $z = e^{-\inp{\lambda-it,x}}$ for all $x \in \X$, we obtain
    \[
        |\varphi_\lambda(t)| = \prod_{x\in \X} \left|\frac{1-e^{-\langle
                \lambda,x\rangle}}{1-e^{-\langle \lambda -
                it,x\rangle}}\right| \leq \exp\left\{\Re\left(\sum_{x\in\X} e^{-\inp{\lambda-it,x}}\right)-\sum_{x\in \X} e^{-\inp{\lambda,x}}\right\}.
    \]
    Since $\Re(z) \leq |z|$ for all complex number $z$, we deduce that
    \begin{equation}
        \label{eq:varphi_bound}
        |\varphi_\lambda(t)| \leq \exp\left\{\frac{1}{|e^\alpha-e^{it_1}|}\frac{1}{|e^\beta-e^{it_2}|}- \frac{1}{e^\alpha - 1}\frac{1}{e^\beta- 1}\right\}.
    \end{equation}
    Let us now describe the set $[-\pi,\pi]^2\setminus\Ecal_n$.
    By the inequality~\eqref{eq:bound_Gamma} of Proposition~\ref{prop:bounds_Gamma}, we know that there exists a positive constant $C$ such that for all $t \in \RR^2$,
    \begin{equation*}
    \|\Gamma_{\lambda_n}^{1/2} t\| = \sqrt{\Gamma_{\lambda_n}(t,t)} \leq C\max\left\{\sqrt{n_1}|t_1|,\frac{n_2}{\sqrt{n_1}} |t_2|\right\}
    \end{equation*}
    Since $L_\lambda = \bigO(n_1^{-1/2})$, we can find some constant
    $c > 0$ such that for all $t \in \RR^2$, the condition $t \notin \Ecal_n$ implies
    \begin{equation}
        |t_1| \geq c \qquad\text{or}\qquad |t_2| \geq c \, \frac{n_1}{n_2}.
    \end{equation}
    In particular, it is enough to bound $|\varphi_\lambda(t)|$ on $\{c \leq
        |t_1| \leq \pi\}$ and $\{c\frac{n_1}{n_2} \leq |t_2| \leq \pi\}$. Without
    loss of generality, we can assume that $c < \pi$. Let us begin with the case $c \leq |t_1| \leq c$. It is easy to check that
     $|e^\alpha-e^{it_1}| \geq |e^\alpha-e^{ic}| \geq
        (e^\alpha-1)$ and $|e^\beta-e^{it_2}| \geq (e^\beta-1)$. Also,
    \[
        \frac{1}{e^{\beta_n}-1} \asymp \frac{n_2}{n_1} \qquad\text{and}\qquad \frac{1}{e^{\alpha_n}-1}-\frac{1}{|e^{\alpha_n}-e^{ic}|} \asymp \frac{(n_1)^2}{n_2}.
    \]
    Hence there exists $C_1 > 0$ such that $|\varphi_\lambda(t)| \leq \exp\{-C_1 n_1\}$ uniformly on $\{c < |t_1| \leq \pi\}$. 
    We use the same method to bound $\varphi_\lambda$ in the domain $\{c\frac{n_1}{n_2} \leq |t_2| \leq \pi\}$, starting with the inequalities $|e^\alpha - e^{it_1}| \geq (e^\alpha-1)$, $|e^\beta-e^{it_2}| \geq |e^\beta-e^{ic\beta}|$, and the estimates
    \[
        \frac{1}{e^{\alpha_n}-1} \asymp \frac{(n_1)^2}{n_2} \qquad \text{and} \qquad \frac{1}{e^{\beta_n}-1} - \frac{1}{|e^{\beta_n}-e^{ic \beta_n}|} \asymp \frac{n_2}{n_1}.
    \]
    Thus there exists $C_2 > 0$ such that $|\varphi_\lambda(t)| \leq \exp\{-C_2 n_1\}$ uniformly on $\{c\frac{n_1}{n_2} \leq |t_2| \leq \pi\}$.

    Therefore, the existence of a positive constant $C$ such that $|\varphi_n(t)| \leq e^{-C n_1}$ for all $t \in [-\pi,\pi]^2\setminus\Ecal_n$ follows. This implies the announced result because $n_1 \to +\infty$ and $\log n_2 = o(n_1)$ under the assumptions of Theorem~\ref{thm:main}.
\end{proof}

\subsection{Proof of the main theorem}

We give a proof of Theorem~\ref{thm:main} in the case $X = \N^2$ of partitions whose parts have non-zero components. Let $(n(k))_k$ be a sequence of vectors in $\N^2$ satisfying the assumptions of the theorem and
consider the sequence of parameters $\lambda_k = (\alpha_k,\beta_k)$, taken as the unique solutions of the implicit equations \eqref{eq:parameters}. Propositions \ref{prop:determinant}, \ref{cor:eigenvalue}, \ref{prop:lyapunov-condition} and \ref{prop:cramer-condition} show that, for the rate
\[
    a_k := \frac{1}{n_2(k)\sqrt{n_1(k)}},
\]
all the assumptions of Proposition~\ref{prop:framework} are satisfied. Therefore, there is a local limit theorem of rate $a_k$ for the random variable $N$ under $\PP_{\lambda_k}$. In particular,
\[
    \PP_{\lambda_k}(N  = n(k)) = \frac{1}{2\pi\sqrt{\det \Gamma_{\lambda_k}}}\exp\left\{-\frac{1}{2}\left\|\Gamma_{\lambda_k}^{-\frac{1}{2}}(n(k) - \EE_{\lambda_k}N)\right\|^2\right\} + \bigO(a_k).
\]
Remember that we chose the parameters in Section~\ref{sec:calibration} to ensure $\|n(k) - \EE_{\lambda_k}N\| = \bigO((n_1)^2/n_2)$. By the bound \eqref{eq:bound_Gamma_inv} of Proposition~\eqref{prop:bounds_Gamma} and the Cauchy-Schwarz inequality, we see therefore that $\|\Gamma_{\lambda_k}^{-1/2}(n(k) - \EE_{\lambda_k}N)\|$ tends to $0$. We then deduce from Proposition~\ref{prop:determinant} and $\sqrt{n_1} \to +\infty$ that
\[
    \PP_{\lambda_k}(N=n(k)) \sim \frac{1}{2\pi}\frac{\beta_k^2}{\sqrt{\Delta(\alpha_k)}},
\]
which, together with the equality~\eqref{eq:link} and the estimate $\log Z_{\lambda_k} = \frac{\Phi(\alpha_k)}{\beta_k} - \frac{\Psi(\alpha_k)}{2} + o(1)$ following from Proposition~\ref{prop:approximation}, implies
\[
    p_{\N^2}(n(k)) \sim \frac{\beta_k^2}{2\pi} \frac{e^{-\frac{1}{2}\Psi(\alpha_k)}}{\sqrt{\Delta(\alpha_k)}}\exp\left\{\alpha_k n_1(k) + \beta_kn_2(k) + \frac{\Phi(\alpha_k)}{\beta_k}\right\}.
\]
We finally use the implicit equations~\eqref{eq:parameters} to simplify this expression.

 \section{A framework for local limit theorems}

 The aim of this section is to provide a general framework as well as mild conditions under which local limit theorems hold for sums of independent random lattice vectors. We focus on Berry-Esseen-like estimates where existence of third moments is assumed and rates of convergence are established. The conditions need to be flexible enough to handle the strong anisotropy that occurs in our problem. Note that this framework also works in the settings of \citet{baez-duarte_hardy-ramanujan_1997,sinai_probabilistic_1994,bogachev_universality_2011}. 
 
 Let $J$ be some countable set. Let $\{\xi_j\}_{j\in J}$ be the canonical process on $(\Z^d)^J$ and consider a sequence of product probability measures $(\PP_k)$ on the product space $(\Z^d)^J$ such that
\[
    \sup_k \sum_{j\in J} \EE_k\|\xi_j\|^2 < \infty.
\]
This condition implies that for all $k$, the series $\sum_j \xi_j$ converges
$\PP_k$-almost surely to a random vector $S$. Moreover, the random vector $S$ has a finite expectation $m_k = \EE_k S$ as well as a finite covariance matrix
$\Gamma_k = \EE_k [(S - m_k)(S-m_k)^\top]$. Let $\sigma_k^2$ be the smallest eigenvalue of $\Gamma_k$. We make the
assumption that $\Gamma_k$ is \emph{non degenerate} (at least for every $k$
large enough), which is equivalent to $\sigma_k > 0$ so that it has a unique symmetric positive-definite square root
$\Gamma_k^{1/2}$, and we write $\Gamma_k^{-1/2}$ for its inverse. Let $g_d(x) = (2\pi)^{-\frac{d}{2}} e^{-\frac{1}{2}\|x\|^2}$ denote the density of the standard normal distribution in $\RR^d$.
 
\begin{definition}
Let $(a_k)$ be a sequence of positive numbers tending to $0$. The sequence $(\PP_k)$ satisfies a (Gaussian) \emph{local limit theorem with rate} $a_k$ if
     \[
         \limsup_{k\to+\infty} \sup_{n\in \Z^d}\; \frac{1}{a_k}\left|\PP_k(S = n) - \frac{g_d(\Gamma_k^{-\frac{1}{2}}(n - m_k))}{\sqrt{\det \Gamma_k}}\right| < \infty.
     \]
\end{definition}

We will give simple sufficient conditions for a local limit theorem to hold when the existence of third moments is assumed, that is
\[
    \sup_{k} \sum_{j\in J} \EE_k\|\xi_j\|^3 < \infty.
\]
Under this assumption we associate to each measure $\PP_k$ a scale-independent quantity $L_k$ analogous to the Lyapunov ratio \citep[p. 59]{bhattacharya_normal_2010}:
\[
    L_k := \sup_{t \in \RR^d\setminus\{0\}} \frac{1}{{\|\Gamma_k^{1/2}t\|}^{3}}\sum_{j\in J} \EE_k|\langle t,\xi_j - \EE_k \xi_j\rangle|^3.
\]
Finally, we consider the ellipsoid $\Ecal_k$ defined by
\[
    \Ecal_k := \left\{t \in \RR^d: \|\Gamma_k^{1/2} t\| \leq \dfrac{1}{4L_k}\right\} = \Gamma_k^{-1/2}\left(\left\{u \in \RR^d : \|u\| \leq \frac{1}{4L_k}\right\}\right).
\]

The following proposition gives three conditions on the product distributions $\PP_k$ that entail a local limit theorem with given speed of convergence. Notice that, at least in the one-dimensional \textit{i.i.d.} case, there is no loss in the rate of convergence (consider for example a sequence of independent Bernoulli variables with parameter $0 < p < \frac{1}{2}$).

 \begin{proposition}
     \label{prop:framework}
     Let $(a_k)$ be a sequence of positive numbers tending to $0$ such that
     \[
         \frac{1}{\sigma_k\sqrt{\det \Gamma_k}} = \bigO(a_k), \quad \frac{L_k}{\sqrt{\det \Gamma_k}} = \bigO(a_k),\quad
         \sup_{t \in [-\pi,\pi]^d\setminus \Ecal_k} \left|\EE_k (e^{i\sca{t,S}})\right| = \bigO(a_k).
        \]
     Then, the sequence $(\PP_k)$ satisfies a local limit theorem with rate $a_k$.
 \end{proposition}

 \begin{proof}
We resort to Fourier analysis in order to bound the quantity
\[
    D_k = (2\pi)^d\,\sup_{n\in \Z^d} \left|\PP_k(S = n) - \frac{g_d(\Gamma_k^{-\frac{1}{2}}(n - m_k))}{\sqrt{\det \Gamma_k}}\right|.
\]
The strategy of the proof is to compare the distribution of normalized random vector $S$ under the measure $\PP_k$ with the normal distribution $\mathcal{N}(m_k,\Gamma_k)$. This is achieved by comparing their respective characteristic functions.
Let $\varphi_k$ be the characteristic function of $S$ under the measure $\PP_k$. By definition, we have for all $t \in \RR^d$, 
\[
    \varphi_k(t) = \EE_k[e^{i\langle t, S\rangle}] = \sum_{n \in \Z^d} \PP_k(S=n) \,e^{i\langle
        t,n\rangle}.
\]
The probabilities $\PP_k(S=n)$ for $n \in \Z^d$ thus appear as the Fourier coefficients of the periodic function $\varphi_k$. In particular, we have an inversion formula:
\[
    \forall n \in \Z^d,\qquad \PP_k(S=n) = \frac{1}{(2\pi)^d} \int_T \varphi_k(t) \, e^{-i\langle t,
    n\rangle}\,dt,
\]
the integral being taken over $T = [-\pi,\pi]^d$.

Now, consider the lattice random vector $Y_k = \Gamma_k^{-1/2}(S-m_k)$. It has zero mean and it is normalized so that its covariance matrix is the identity matrix. Let $\psi_k$ denote the characteristic function of $Y_k$. By definition, we have $\psi_k(t) = \EE_k(e^{i\langle t, Y_k\rangle})$ for all $t \in \RR^d$. Notice that the functions $\varphi_k$ and $\psi_k$ are linked together by the identity $\psi_k(\Gamma_k^{1/2}t) = \varphi_k(t) e^{-i\sca{t,m_k}}$. Hence for every $n\in \Z^d$, one has
\begin{equation}
    \label{eq:proba-fourier}
    \PP_k(S=n) = \frac{1}{(2\pi)^d}\int_T
    \psi_k(\Gamma_k^{1/2}t)\,e^{-i\langle t,n - m_k\rangle}\,dt.
\end{equation}

We turn to the second term in $D_k$, corresponding to the density of the normal distribution $\mathcal{N}(m_k,\Gamma_k)$. The Fourier inversion formula yields for all $n \in \Z^d$,
\begin{equation}
    \label{eq:density-fourier}
    \frac{g_d(\Gamma_{k}^{-1/2}(n
    - m_{k}))}{\sqrt{\det \Gamma_k}} = 
    \frac{1}{(2\pi)^d}\int_{\RR^d} e^{-\frac{1}{2}\|\Gamma_k^{1/2}t\|^2}e^{-i\langle
        t,n-m_k\rangle}\,dt,
\end{equation}
so that equations \eqref{eq:proba-fourier} and \eqref{eq:density-fourier} imply together that
\[
    D_k = \sup_{n\in \Z^d} \left|\int_T
    \psi_k(\Gamma_k^{1/2}t)\,e^{-i\langle t,n - m_k\rangle}\,dt - \int_{\RR^d} e^{-\frac{1}{2}\|\Gamma_k^{1/2}t\|^2}e^{-i\langle t,n-m_k\rangle}\,dt \right|.
\]
We split the domain of integration according to the partition $(T\setminus\Ecal_k) \cup (T\cap \Ecal_k) \cup (\RR^d \setminus (T\cup \Ecal_k))$ of $\RR^d$ and we use the triangular inequality:
\begin{equation*}
    \label{eq:splitting}
    D_k \leq \int_{T\setminus \Ecal_k} |\psi_k(\Gamma_k^{1/2}t)|\,dt + \int_{T\cap\Ecal_k} \left|\psi_k(\Gamma_k^{1/2} t)
        - e^{-\frac{1}{2}\|\Gamma_k^{1/2}t\|^2}\right|\,dt
    + \int_{\RR^d \setminus (T\cap\Ecal_k)} e^{-\frac{1}{2}\|\Gamma_k^{1/2}t\|^2}\,dt .
\end{equation*}
Because of the assumption on $\varphi_k(u)=\psi_k(t)$ in the bounded domain $T\setminus \Ecal_k$, the contribution of the first term of the right-hand side is $\bigO(a_k)$. The two other terms are respectively handled by Lemma~\ref{lem:central} and Lemma~\ref{lem:tails} below.
\end{proof}

\begin{lemma}[Central approximation]
    \label{lem:central}
    Under the assumption $L_k = \bigO(a_k \sqrt{\det \Gamma_k})$ of Proposition~\ref{prop:framework},
    \[
        \limsup_{k\to\infty} \frac{1}{a_k}\int_{\Ecal_k} \left|\psi_k(\Gamma_k^{1/2} t)
        - e^{-\frac{1}{2}\|\Gamma_k^{1/2}t\|^2}\right|\,dt < \infty.
    \]
\end{lemma}

\begin{proof}
After the substitution $u = \Gamma_k^{1/2} t$, and because $\|u\|^3e^{-\frac{1}{3}\|u\|^2}$ is integrable on $\RR^d$, we see that we need only prove the following inequality in the domain $\|u\| \leq \frac{1}{4}L_k^{-1}$:
    \begin{equation}
        \label{eq:petrov}
        \left|\psi_k(u) - e^{-\frac{1}{2}\|u\|^2}\right| \leq 16 L_k \|u\|^3 e^{-\frac{1}{3}\|u\|^2},
    \end{equation}
    
    We now turn to the proof of \eqref{eq:petrov}. For all $j \in J$, let $\xi_j'$ be an independent copy of $\xi_j$. Then $\xi_j - \xi_j'$ has zero mean, its second moments are twice those of the centered random variable $\overline{\xi}_j := \xi_j - \EE_k \xi_j$, and $\EE_k|\inp{t,\xi_j-\xi_j'}|^3 \leq 8 \EE_k|\inp{t,\overline{\xi}_j}|^3$. Using  $|\EE_k[e^{i\inp{t,\overline{\xi}_j}}]|^2 = \EE_k[e^{i\inp{t,\xi_j-\xi_j'}}]$ and  the classical Taylor expansion estimate of the characteristic function for $\xi_j - \xi_j'$, we thus have for all $t \in \RR^d$,
\[
        |\EE_k[e^{i\inp{t,\overline{\xi}_j}}]|^2 \leq 1 - \frac{2}{2!}\EE_k|\inp{t,\overline{\xi}_j}|^2 + \frac{8}{3!}\EE_k|\inp{t,\overline{\xi_j}}|^3.
\]
Since $1 - x \leq e^{-x}$ for all $x \in \RR$, we deduce
\begin{equation}
    \label{eq:wj_bound}
    |\EE_k[e^{i\inp{t,\overline{\xi}_j}}]| \leq \exp\left\{-\frac{1}{2}\EE_k|\inp{t,\overline{\xi}_j}|^2 + \frac{2}{3}\EE_k|\inp{t,\overline{\xi}_j}|^3\right\},
\end{equation}
    so that, for all $u = \Gamma_k^{1/2} t$ satisfying $\|u\| \leq \frac{1}{4}L_k^{-1}$, the definitions of $\Gamma_k$ and $L_k$ imply
    \begin{equation}
        \label{eq:psik_bound}
        |\psi_k(u)|^2 = \prod_{j\in J} |\EE_k[e^{i\inp{t,\overline{\xi}_j}}]|^2 \leq \exp\left\{-\|\Gamma_k^{1/2}t\|^2 + \frac{4}{3}L_k \|\Gamma_k^{1/2}t\|^3\right\} \leq \exp\left\{-\frac{2}{3}\|u\|^2\right\}.
    \end{equation}
Let us begin with the case $\frac{1}{2}L_k^{-1/3} \leq \|u\| \leq \frac{1}{4}L_k^{-1}$. 
    In this domain, $\|u\| \geq \frac{1}{2}L_k^{-1/3}$ implies $L_k\|u\|^3 \geq 8$. Using $|\psi_k(u)-e^{-\frac{1}{2}\|u\|^2}| \leq |\psi_k(u)|+e^{-\frac{1}{2}\|u\|^2}$ and $e^{-\frac{1}{2}\|u\|^2} \leq e^{-\frac{1}{3}\|u\|^2}$, we see that \eqref{eq:petrov} holds.
    
    We continue with the remaining case: $\|u\| \leq \frac{1}{4}L_k^{-1}$ and $\|u\| \leq \frac{1}{2}L_k^{-1/3}$. For all $j \in J$, let $v_j(t) = \exp\{-\frac{1}{2}\EE_k|\inp{t,\overline{\xi}_j}|^2\}$ and $w_j(t) = \EE_k[e^{i\inp{t,\overline{\xi}_j}}]$. By \eqref{eq:wj_bound} we see that $0 < v_j(t) < |w_j(t)|$, so the following elementary inequality holds:
    \begin{equation}
        \label{eq:feller_ineq}
        \left|\prod_{j\in J} w_j(t) - \prod_{j\in J} v_j(t)\right| \leq \left(\prod_{j\in J} |w_j(t)|\right) \sum_{j \in J} \frac{|w_j(t) - v_j(t)|}{v_j(t)}.
    \end{equation}
    We proved in \eqref{eq:psik_bound} that the product in the right-hand side of \eqref{eq:feller_ineq} is bounded by $\exp\{-\frac{1}{3}\|u\|^2\}$. By Jensen's inequality and the definition of $L_k$, the condition $\|u\| \leq \frac{1}{2}L_k^{-1/3}$ implies that
    \[
        v_j(t) \geq \exp\left\{-\frac{1}{2}\left(\EE_k|\inp{t,\overline{\xi}_j}|^3\right)^{2/3}\right\} \geq \exp\left\{-\frac{1}{2}(L_k\|u\|^3)^2\right\} \geq \exp\left\{-\frac{1}{8}\right\} > \frac{1}{2}.
    \]
    We estimate the summand in the right-hand side of \eqref{eq:feller_ineq} using Taylor expansions of $w_j(t)$ and $v_j(t)$. Since $v_j(t) \geq 1/2$,
    \begin{align*}
        \frac{|w_j(t) - v_j(t)|}{v_j(t)} & \leq 2\left|w_j(t) - 1 + \frac{1}{2}\EE_k|\inp{t,\overline{\xi}_j}|^2\right| + 2\left|v_j(t) - 1 + \frac{1}{2}\EE_k|\inp{t,\overline{\xi}_j}|^2\right|\\
                                         & \leq \frac{1}{3}\EE_k|\inp{t,\overline{\xi}_j}|^3 + \frac{1}{4}\left(\EE_k|\inp{t,\overline{\xi}_j}|^2\right)^2 \leq \frac{1}{3}\EE_k|\inp{t,\overline{\xi}_j}|^3 + \frac{1}{4}(\EE_k|\inp{t,\overline{\xi}_j}|^3)^{4/3}\\
                                         & \leq \EE_k|\inp{t,\overline{\xi}_j}|^3.
    \end{align*}
    Summing up for all $j \in J$, we obtain $|\psi_k(u)-e^{-\frac{1}{2}\|u\|^2}|\leq L_k \|u\|^3e^{-\frac{1}{3}\|u\|^2}$ and \eqref{eq:petrov} follows.

\end{proof}

\begin{lemma}[Tails completion]
    \label{lem:tails}
    Under the assumptions $L_k = \bigO(a_k\sqrt{\det \Gamma_k})$ and $\sigma_k^{-1} = \bigO(a_k\sqrt{\det \Gamma_k})$ of Proposition~\ref{prop:framework}, 
    \begin{equation}
        \label{eq:tails_bound}
        \limsup_{k\to\infty} \frac{1}{a_k}\int_{\RR^d\setminus(T\cap \Ecal_k)} e^{-\frac{1}{2}\|\Gamma_k^{1/2}t\|^2} \,dt < \infty.
    \end{equation}
\end{lemma}

\begin{proof}
    The domain of integration splits into $\RR^d\setminus (T\cap \Ecal_k) = (\RR^d\setminus\Ecal_k) \cup (\RR^d\setminus T)$ and we deal separately with the two sub-domains, based on the inequality 
    \begin{equation}
        \label{eq:tails_split}
        \int_{\RR^d\setminus\Ecal_k} e^{-\frac{1}{2}\|\Gamma_k^{1/2}t\|^2} \,dt
        \leq \frac{1}{\sqrt{\det \Gamma_k}}\left(\int_{\{\|u\| > \frac{1}{4}L_k^{-1}\}}
        e^{-\frac{1}{2}\|u\|^2}\,du  + \int_{\{\Gamma_k^{-1/2}u \notin T\}} e^{-\frac{1}{2}\|u\|^2}\,du\right).
    \end{equation}
    Let us begin with the first summand. After the polar substitution $(r,\hat{u}) \in (0,+\infty)\times \mathbb{S}^{d-1} \mapsto u = r\hat{u}$, where $\mathbb{S}^{d-1}$ is the unit sphere of $\RR^d$, we see that it is proportional (up to the surface area of $\mathbb{S}^{d-1}$) to
    \[
        \int_{r > \frac{1}{4}L_k^{-1}} r^{d-1}e^{-\frac{1}{2}r^2}\,dr \leq 4L_k \int_0^\infty r^d e^{-\frac{1}{2}r^2}\,dr.
    \]
    Since the latter integral is finite and $L_k = \bigO(a_k\sqrt{\det \Gamma_k})$, the first summand of \eqref{eq:tails_split} yields a finite contribution in \eqref{eq:tails_bound}.

    In order to deal with the second summand, let us remark that $\Gamma_k^{-1/2}u \notin T$ implies $\|u\| > \sigma_k\pi$ so that the rest of the proof is entirely similar to the first part, except that it uses the assumption on $\sigma_k^{-1}$.   
\end{proof}

\section*{Acknowledgements}
The author would like to thank Nathana\"el Enriquez for his supervision and useful discussions about this work, as well as the anonymous referee whose careful reading and comments have helped improve the presentation of the paper.

\bibliographystyle{amsplnat}
\bibliography{references}

\end{document}